\documentclass{amsart}
\usepackage{amssymb}
\usepackage{amsmath}
\usepackage{amsthm}
\usepackage[latin1]{inputenc}
\usepackage[T1]{fontenc}
\usepackage{hyperref}
\usepackage[all]{xy}
\usepackage{enumerate}
\usepackage{pstricks}
\usepackage{pstricks-add}

\theoremstyle{plain}
\newtheorem{theorem}{Theorem}[section]

\newtheorem{proposition}[theorem]{Proposition}

\theoremstyle{definition}
\newtheorem{definition}[theorem]{Definition}

\newtheorem*{notation}{Notation}
\newtheorem{remark}[theorem]{Remark}

\numberwithin{equation}{section}

\newcommand{\Cs}{\ensuremath{\mathrm{C}^*}}
\newcommand{\Csr}{\ensuremath{\mathrm{C}_r^*}}
\newcommand{\E}{\ensuremath{\mathcal{E}}}

\DeclareMathOperator{\re}{\R e}
\newcommand{\Fp}{\mathbb{F}_p}
\newcommand{\Q}{\mathbb{Q}}
\newcommand{\R}{\mathbb{R}}
\newcommand{\C}{\mathbb{C}}

\newcommand{\scal}[2]{\langle #1,#2\rangle}

\def\bn{\bar{n}}

\DeclareMathOperator{\F}{\mathcal{F}}
\DeclareMathOperator{\U}{\mathcal{U}}
\DeclareMathOperator{\I}{\mathcal{I}}

\newcommand{\Nbar}{\ensuremath{\bar{N}}}
\newcommand{\Xbar}{\ensuremath{\bar{X}}}

\newcommand{\Stz}{\mathcal{S}}

\begin{document}

\title[$\Cs$-algebraic intertwiners]{$\Cs$-algebraic intertwiners for degenerate principal series of special linear groups}
\author{Pierre Clare}
\address{Pierre Clare\\ The Pennsylvania State University\\ Department of Mathematics\\ McAllister Building\\ University Park, PA - 16802\\USA}
\email{clare@math.psu.edu}

\date{\today}

\setcounter{tocdepth}{3}

\begin{abstract}
We construct unitary intertwiners for degenerate $\Cs$-algebraic universal principal series of $\mathrm{SL}(n+1)$ over a local field by explicitely normalizing standard intertwining integrals a the level of Hilbert modules.
\end{abstract}
\maketitle


\section{Introduction}

The relation between the theory of unitary representations of topological groups and $\Cs$-algebras is a classical topic (see \cite{Dixmier}) that knew many developements over the years. An example of particular relevance here is the work of M. A. Rieffel who gave in \cite{RieffelIRC*A} a description of induced representations in terms of Hilbert modules.

In an effort to obtain elements of a $\Cs$-algebraic formulation of the representation theory of semisimple Lie groups, Rieffel's construction was adapted in \cite{PArtmodules} to describe parabolic induction, which resulted in the construction of Hilbert modules related to the principal series induced from a given association class of parabolic subgroups. After the results of generic irreducibility obtained in \cite{PArtmodules} in the spirit of F. Bruhat's theory, the next step in the direction of an accurate description of the $\Cs$-algebras related to a Lie group consists in developping the theory of intertwining operators. These objects were originally introduced by A. W. Knapp and E. M. Stein in \cite{KS1,KS2} (see also \cite{VoganWallach} and recently \cite{ClercEntrelac}) and play a central role in the theory, as they allow to decide when principal series representations are reducible but also relate to the densities in the Plancherel formula and the construction of complementary series.

The central piece in the classical study of intertwining operators is the normalization process, allowing to go from integral transformations analogous to Radon transforms and called \textit{standard intertwining integrals} to unitary intertwiners, that can often be expressed in terms of classical geometric transforms (see for instance \cite{Kobayashi_small_GL}, \cite{PasquOlaf}, \cite{PArtSpnC} for recent examples).

Standard intertwining integrals were studied in the context of Hilbert modules in \cite{PArtmodules} and a result of normalization was obtained in this framework for $\mathrm{SL}(2)$ in \cite{PC*entrelacSL2}, using Fourier transforms twisted by the non-trivial Weyl element. The point of the present paper is to extend this result to the case of maximally degenerate principal series of special linear groups: we will show how a Fourier transform defined on an appropriate space extends to a unitary intertwining operator between Hilbert modules and normalizes the standard intertwining integral in a sense that will be made precise.

The paper is organised as follows: Section \ref{secE(G/N)} is devoted to the description of degenerate principal series of $\mathrm{SL}(n+1)$ in terms of Hilbert modules. More precisely, universal $\Cs$-algebraic principal series are introduced in the discussion concluded by Definition \ref{defE(G/N)} and Proposition \ref{charactE(G/N)} characterises a useful submodule of functions. Section \ref{secintertw} starts with the study of the Radon transform defining the standard intertwining integral. Then we introduce a Fourier transform adapted to the situation and prove in Theorem \ref{thmnormal} that it extends to a unitary operator between Hilbert modules that normalizes the standard integral at the level of $\Cs$-algebraic principal series. Finally, we explain in Section \ref{secnonarchi} how the results extend to the case of non-archimedean local fields.

The results presented here are part of an ongoing joint research project with members of the MAPMO (University of Orléans): A. Alvarez, P. Julg and V. Lafforgue. This article greatly benefited from their help and comments.

\section{\texorpdfstring{$\Cs$-algebraic universal principal series}{C*-algebraic universal principal series}}\label{secE(G/N)}

\subsection{Structure and notations}

Let $F$ be the field $\R$ or $\C$ of real or complex numbers. If $x$ is a matrix with coefficients in $F$, the matrix obtained by conjugating all the entries of $x$ will be denoted by $\bar{x}$. The real matrix $x=\frac{1}{2}\left(x+\bar{x}\right)$ will be called the \textit{real part} of $x$ and denoted by $\re(x)$.

Troughout the article, $G_F$ will denote the group $\mathrm{SL}(n+1,F)$ of matrices of size $n+1$ with determinant $1$, for $n\geq1$. Let $\theta$ be the Cartan involution of $G_F$ defined by $\theta(g)={^t}\bar{g}^{-1}$. The subgroup $P_F$ of upper block-triangular matrices of type $(n,1)$ is a maximal parabolic subgroup of $G_F$. It admits a Langlands decomposition $P_F = L_F N_F$ with $\theta$-stable Levi component

\[L_F=\left\{\left[\begin{array}{ccc|c}&&&0\\&a&&\vdots\\&&&0\\\hline0&\cdots&0&\det(a)^{-1}\end{array}\right]\:,\: a\in\mathrm{GL}(n,F)\right\}\]

and unipotent radical \[N_F=\left\{\left[\begin{array}{ccc|c}&&&\\&I_n&&X\\&&&\\\hline0&\cdots&0&1\end{array}\right]\:,\: X\in F^n\right\}\] where $I_n$ denotes the identity matrix of size $n$, so that $P_F$ identifies with the semi-direct product $\mathrm{GL}(n,F)\ltimes F^n$. The opposite parabolic subgroup $\bar{P}_F$ is the image of $P_F$ under $\theta$. It decomposes as $\bar{P}_F=L_F \Nbar_F$, with $\Nbar_F={^t}N_F$.

The group $L_F\simeq\mathrm{GL}(n,F)$ is endowed with the Haar measure $d^\times a$ defined by \[d^\times a = \left|\det(a)\right|^{-n}da\] where $da$ denotes the restiction of the Lebesgue measure of $F^{n^2}$ and $|\cdot|$ is the usual absolute value in the real case and defined by $|z|=\bar{z}z$ in the complex case.

\subsection{\texorpdfstring{The homogeneous spaces $G/N$ and $G/\Nbar$}{The homogeneous spaces G/N and G/Nbar}}

The central result of this paper will rely on some analysis on the quotient space $G/N$ described in this section. To avoid confusion, group actions will be denoted using dots (\textit{e.g.} $a.b$), whereas matrix multiplications will be written without any symbol (\textit{e.g.} $ab$).

\begin{notation}The vector space of matrices with $p$ lines, $q$ columns and coefficients in $F$ is denoted by $\mathrm{M}_{p,q}(F)$ (simply $\mathrm{M}_p(F)$ if $p=q$) and we let $X_F=\mathrm{M}_{n+1,n}(F)$.\end{notation}

The groups $G_F$ and $L_F$ act on $X_F$ from the left and the right respectively by $g.x=gx$ and $x.a=xa$ for $g$ in $G_F$, $a$ in $\mathrm{GL}(n,F)$ and $x$ in $X_F$.


\begin{notation} The set of elements of rank $n$ in $X_F$ is denoted by $\mathrm{M}^\times_{n+1,n}(F)$.\end{notation}

This space is dense in $X_F$ and inheritates the actions of $G_F$ and $\mathrm{GL}(n,F)$. We denote by $dx$ the $G_F$-invariant measure obtained by restricting the Lebesgue measure of $\mathrm{M}_{n+1,n}(F)\simeq F^{n(n+1)}$. The right action of $\mathrm{GL}(n,F)$ on $X_F$ transforms the measure according to \begin{equation}\label{LmesX}\int_{X_F}f(x.a)\,dx=\left|\det\nolimits_F(a)\right|^{-(n+1)}\int_{X_F}f(x)\,dx.\end{equation}

\begin{proposition}\label{G/N}
The homogeneous space $G_F/N_F$ identifies with $\mathrm{M}^\times_{n+1,n}(F)$ as a topological and measured space. The identification is $(G_F,L_F)$-equivariant.
\end{proposition}

\begin{proof} The stabiliser subgroup in $G_F$ of $x_0=\left[\begin{array}{ccc}&&\\&I_n&\\&&\\\hline0&\cdots&0\end{array}\right]$ is $N_F$, so the map $b:G_F\rightarrow X_F\;,\;g\mapsto g.x_0$ identifies $G_F/N_F$ with $b(G)=G.x_0$ in a $G_F$-equivariant way. Since $L_F$ normalizes $N_F$ and \[\left[\begin{array}{ccc|c}&&&0\\&a&&\vdots\\&&&0\\\hline0&\ldots&0&\det(a)^{-1}\end{array}\right]x_0 = \left[\begin{array}{ccc}&&\\&a&\\&&\\\hline0&\cdots&0\end{array}\right] = x_0.a\] for any $a$, it is also $L$-equivariant.

To identify the image of $b$, we observe that it is the map that extracts the left block of size $(n+1,n)$ of a matrix in $\mathrm{SL}(n+1,F)$, that is
\[b:\left[\begin{array}[c]{ccc|c}&&&*\\&^{\displaystyle u}&&\vdots\\&&&*\end{array}\right]\longmapsto u\]
for $u\in X_F$, so $b(g)$ has maximal rank for all $g$. Surjectivity follows from the fact that any linearly independant family may be completed into a basis with determinant $1$. Finally, $b$ is compatible with the topologies induced on $G_F$ and $\mathrm{M}_{n+1,n}^\times(F)$ by the standard ones on $\mathrm{M}_{n+1}(F)$ and $X_F$, and $G_F/N_F$ admits essentially one $G_F$-invariant measure, which corresponds \textit{via} $b$ to the restriction of the Lebesgue measure to $X_F$.
\end{proof}

From now on, we will use freely the identification established in the above result and identify any element of $L_F$ with the corresponding element of $\mathrm{GL}(n,F)$.

An analogous description can be obtained for the `opposite' quotient $G_F/\Nbar_F$ identified with $\Xbar_F=\mathrm{M}_{n,n+1}(F)$ through the map $\bar{b}:g\longmapsto y_0g^{-1}$ where the matrix $y_0=\left[\begin{array}{ccc|c}&&&0\\&I_n&&\vdots\\&&&0\end{array}\right]$ has stabiliser $\Nbar_F$ and $\Xbar_F$ carries actions of $G_F$ and $L_F$ respectively given by $g.y=yg^{-1}$ and $y.a=a^{-1}y$ for $g$ in $G_F$ and $a$ in $L_F$.

\subsection{\texorpdfstring{The modules and $\E(G/N)$ and $\E(G/\Nbar)$}{The modules and E(G/N) and E(G/Nbar)}}\label{E(G/N)sec}

From now on, we will systematically omit the subscript $_F$. The set of smooth functions with compact support on a space $T$ will be denoted by $C_c^\infty(T)$, and the Schwartz space of functions on $T$ all of whose derivatives are rapidly decreasing by $\Stz(T)$.

In order to study intertwining operators in a $\Cs$-algebraic context, we will work in the framework of universal principal series, introduced in \cite{PArtmodules}. The basic objects will be Hilbert modules defined as completions of $C_c^\infty(G/N)$ and $C_c^\infty(G/\Nbar)$ respectively.

\begin{notation} If $f$ is a function of the variable $x$ and $m$ a group element such that $mx$ makes sense, the right translate of $f$ by $m$ is denoted by $f^m:x\mapsto f(xm)$. Similarly, ${^m}f(x)=f(mx)$.\end{notation}

The actions of $G$ and $L$ can be defined for functions on the larger space $X$. Namely, for a function $f$ defined on $X$, and elements $g$ of $G$ and $a$ of $L$, we let
\begin{equation}\label{actionG}g.f(x)=f(g^{-1}x)\end{equation}
\begin{equation}\label{actionL}f.a(x) = \left|\det(a)\right|^{-\frac{n+1}{2}}f(xa^{-1}) = \left|\det(a)\right|^{-\frac{n+1}{2}}f^{a^{-1}}(x),\end{equation}
which integrates to
\begin{equation}\label{actionLint}\tag{\ref{actionL}'}f.\varphi(x) = \int_{\mathrm{GL}(n,F)}f(xa^{-1})\varphi(a)\left|\det(a)\right|^{-\frac{n+1}{2}}d^\times a\end{equation}
for any $\varphi$ integrable over $\mathrm{GL}(n,F)$.

Finally we consider the pairing
\begin{equation}\label{scal}\scal{f}{h}_X(a) = \left|\det(a)\right|^{\frac{n+1}{2}}\int_X\overline{f(x)}h(xa)\,dx = \left|\det(a)\right|^{\frac{n+1}{2}}\scal{f}{h^a}_{L^2}\end{equation}
for suitable functions $f$ and $h$ defined on $X$, where $\scal{\cdot}{\cdot}_{L^2}$ denotes the usual inner product for square-integrable functions. If $f$ and $h$ are smooth and compactly supported on $G/N\subset X$, then $\scal{f}{h}_X$ is smooth and compactly supported on $L$ and the above formulas define on $C^\infty_c(G/N)$ a structure of pre-Hilbert module over $C^\infty_c(L)$. The scalar factor involved in \eqref{actionL} shows that $f$ is actually considered as a half-density over $G/N$, and \eqref{scal} then appears as a natural pairing.

The above formulas are special cases of the ones involved in the general construction of parabolic induction modules described in \cite{PArtmodules}. Therefore, they can be used to define a Hilbert module structure as follows:

\begin{definition}\label{defE(G/N)}The Hilbert module obtained from $C^\infty_c(G/N)$ by extending the action \eqref{actionLint} to the reduced $\Cs$-algebra of $\mathrm{GL}(n,F)$, then completing with respect to the norm induced by \eqref{scal} is called the \emph{reduced $\Cs$-algebraic universal principal series} associated to the pair $(G,P)$ and denoted by $\E(G/N)$.\end{definition}

As explained in \cite{PArtmodules}, this construction generalises Rieffel's construction of induction modules \cite{RieffelIRC*A}. It provides a global picture of principal series induced from $P$ in the sense that there exists a $G$-equivariant unitary specialisation map between the tensor product of $\E(G/N)$ with the carrying space of any appropriate element of $\widehat{L}_r$ and the corresponding principal series representation \cite[Corollary 3.5]{PArtmodules}.

The next result allows us to consider Schwartz functions on $X$ as a subspace of $\E(G/N)$. It was suggested to me by Vincent Lafforgue.

\begin{proposition}\label{charactE(G/N)}
The inclusion map $\iota:C_c^\infty(\mathrm{M}^\times_{n+1,n}(F))\longrightarrow\E(G/N)$ extends continuously to $\Stz(X)$.
\end{proposition}

\begin{proof}
According to \cite[Chap.VII \S3]{KnappBeyond}, we can use the $KAK$ decomposition to write any element of $\mathrm{GL}(n,F)$ as a product $k_1ak_2$, where $k_1$ and $k_2$ are elements in a maximal compact subgroup of $\mathrm{GL}(n,F)$ and $a$ is a diagonal matrix with positive entries $a_1,\ldots,a_n$. Then for $f$ and $h$ in $\Stz(X)$, we prove that there exists a constant $C$ such that:
\begin{equation}\label{estimscal}\left|\scal{f}{h}_X\left(k_1ak_2\right)\right|\leq C\prod_{i=1}^n\left|\min(a_i,a_i^{-1})\right|^{\frac{n+1}{2}}.\end{equation}
By polarisation, we may assume that $f=h$. Moreover, since $K$ is compact it is enough to treat the case $k_1=k_2=1$. Identifying $X$ to $\left(F^{n+1}\right)^n$ \textit{via} $x=(x_1,\ldots,x_n)$, the assumption that $f$ is rapidly decreasing implies the existence of positive bounded integrable functions $f_i$ on $F^{n+1}$ such that $\left|f(x)\right|\leq f_1(x_1)\ldots f_n(x_n)$. Observe that $a$ acts on $x$ by multiplying each $x_i$ by $a_i$, so that it is enough to consider the case of one $f_i$. One has, for $f$ positive bounded and integrable on $F^{n+1}$ and $a>0$, \begin{eqnarray*}\left|a\right|^{\frac{n+1}{2}}\int_{F^{n+1}}f(x)f(ax)\,dx&\leq&\left|a\right|^{\frac{n+1}{2}}\left\|f\right\|_\infty\int_{F^{n+1}}f(ax)\,dx\\&\leq&\left\|f\right\|_\infty\left\|f\right\|_1|a|^{-\frac{n+1}{2}}.\end{eqnarray*} Since \[\left|a\right|^{\frac{n+1}{2}}\int_{F^{n+1}}f(x)f(ax)\,dx=\left|a\right|^{-\frac{n+1}{2}}\int_{F^{n+1}}f(a^{-1}x)f(x)\,dx,\] the same computation yields \[\left|a\right|^{\frac{n+1}{2}}\int_{F^{n+1}}f(x)f(ax)\,dx\leq\left\|f\right\|_\infty\left\|f\right\|_1|a|^{\frac{n+1}{2}},\] so that \[\left|a\right|^{\frac{n+1}{2}}\int_{F^{n+1}}f(x)f(ax)\,dx\leq\left\|f\right\|_\infty\left\|f\right\|_1\left|\min\left(a,a^{-1}\right)\right|^{\frac{n+1}{2}},\] which implies \eqref{estimscal}. As a consequence, we will prove that $\scal{f}{h}_X$ belongs to the reduced $\Cs$-algebra of $\mathrm{GL}(n,F)$.

Following \cite[\S4.4]{LafforgueInventiones}, we consider the generalised Harish-Chandra Schwartz space associated to a real reductive group $L$ consisting of functions $F$ on $L$ such that for any $p$, there exists a constant $C_p$ such that the inequality \begin{equation}\label{schwartzm}\left|F(k_1ak_2)\right|\leq C_p\frac{e^{-\rho\log a}}{\left(1+\left\|\log a\right\|\right)^p}\end{equation} is satisfied for all $k_1ak_2$ in the $KAK$ decomposition of $L$, where $\rho$ denotes the half sum of positive roots for $(\mathfrak{l}:\mathfrak{a})$. We will use the fact, whose proof is recalled in \cite{LafforgueInventiones}, that for $p$ large enough, there exists $C_p$ such that \eqref{schwartzm} defines a norm that yields a Banach algebra which is a subalgebra of $\Csr(G)$. In the case at hand, denoting $d_F=\dim_\R(F)$, one has \begin{eqnarray*}2\rho\log a&=&\sum_{i=1}^n(n+1-2i)d_F\log a_i\\&=&(n-1)d_F\log a_1+(n-3)d_F\log a_2+\ldots-(n-1)d_F\log a_n,\end{eqnarray*} so that 
\begin{align}\label{ineq}e^{-\rho\log a}&=\prod_{i=1}^n \left|a_i\right|^{i-\frac{n+1}{2}}=\left(\left|a_1\right|^{-1}\right)^{\frac{n-1}{2}} \left(\left|a_2\right|^{-1}\right)^{\frac{n-3}{2}}\ldots \left|a_{n-1}\right|^\frac{n-3}{2}\left|a_n\right|^\frac{n-1}{2}\\
&\geq\prod_{i=1}^n\left|\min\left(a_i,a_i^{-1}\right)\right|^{\frac{n-1}{2}}\geq\prod_{i=1}^n\left|\min\left(a_i,a_i^{-1}\right)\right|^{\frac{n+1}{2}}\nonumber\end{align}
since $\frac{n+1}{2}>\frac{n-1}{2}$ and $\min\left(a_i,a_i^{-1}\right)\leq1$ for all $i$. Therefore, \eqref{estimscal} implies that $\scal{f}{h}_X$ satisfies \eqref{schwartzm} for any $p$, hence defines an element of $\Csr\left(\mathrm{GL}(n,F)\right)$.

Next we prove that functions in $\Stz(X)$ can be approximated by elements of $C^\infty_c(G/N)$ for the $\Csr\left(\mathrm{GL}(n,F)\right)$-valued norm used to build the Hilbert module $\E(G/N)$. Let $\mu$ denote the Lebesgue measure of $X$. Since $\mu\left(X\setminus\mathrm{M}^\times_{n+1,n}(F)\right)=0$, there exists a sequence $\left\{T_m\right\}_{m\geq1}$ of open subsets of $X$ containing $X\setminus\mathrm{M}^\times_{n+1,n}(F)$, satisfying $\overline{T_{m+1}}\subset T_m$ and such that $\lim\limits_{m\to\infty}\mu(T_m)=0$. Let $\left\{B_m\right\}_{m\geq1}$ be the family of closed balls centered at the origin in $X$ with radius $m$ for the usual norm and consider a sequence $\left\{\chi_m\right\}_{m\geq1}$ of compactly supported non-negative smooth functions on $X$ such that \[\begin{array}{ccl}\chi_m&\equiv&1\quad\text{on $B_m\cap\left(X\setminus T_m\right)$}\\&\equiv&0\quad\text{ on $T_{m+1}$}.\end{array}\]
Let $f$ be in $\Stz(X)$. For all $m$, the pointwise product $f_m=f.\chi_m$ is compactly supported away from singular matrices, hence defines a function in $C^\infty_c(G/N)$. We will prove that the $\Csr\left(\mathrm{GL}(n,F)\right)$-valued norm of $f-f_m$ converges to 0, so that $f$ can be viewed as the limit in $\E(G/N)$ of the sequence $\left\{f_m\right\}_{m\geq1}$. To do so, we denote $\varphi_m=\scal{f-f_m}{f-f_m}_X$. By the same argument used above, \eqref{estimscal} implies that $\varphi_m$ belongs to the generalised Harish-Chandra Schwartz space of $\mathrm{GL}(n,F)$. Furthermore, \eqref{estimscal} and \eqref{ineq} imply that \[\left|\varphi_m\left(k_1ak_2\right)\frac{\left(1+\left\|\log a\right\|\right)^p}{e^{-\rho\log a}}\right|\leq C\left(1+\left\|\log a\right\|\right)^p\prod_{i=1}^n\left|\min\left(a_i,a_i^{-1}\right)\right|\] for some constant $C$ that can be chosen independently of $m$ by the construction of $f_m$. It follows that for any $\varepsilon>0$ there is a compact $Z$ in $\mathrm{GL}(n,F)$ such that \[\sup_{k_1ak_2\notin Z}\left|\varphi_m\left(k_1ak_2\right)\frac{\left(1+\left\|\log a\right\|\right)^p}{e^{-\rho\log a}}\right|\leq\varepsilon.\] By definition of $f_m$, the sequence $\left\{\varphi_m\right\}_{m\geq1}$ converges to 0 uniformly on any compact subset of $\mathrm{GL}(n,F)$ such as $Z$. Therefore, \[\sup_{k_1ak_2\in \mathrm{GL}(n,F)}\left|\varphi_m\left(k_1ak_2\right)\frac{\left(1+\left\|\log a\right\|\right)^p}{e^{-\rho\log a}}\right|\leq\varepsilon\] for $m$ large enough, so that $\varphi_m$ converges to 0 in the generalised Harish-Chandra Schwartz space of $\mathrm{GL}(n,F)$, hence in $\Csr\left(\mathrm{GL}(n,F)\right)$, which concludes the proof.\end{proof}

\begin{remark}\label{remE(G/N)}The advantage of $\Stz(X)$ compared to $C^\infty_c(G/N)$ as a submodule of $\E(G/N)$ is the fact that, unlike compactness of the support, the properties of asymptotic decay of elements in $\Stz(X)$ will be preserved by the Fourier transform used in the normalization process of Section \ref{normalsubsec}, thus allowing to properly define an operator between Hilbert modules.\end{remark}

\begin{remark}
Starting from the pre-Hilbert module structure on $C^\infty_c(G/N)$, it is possible to consider various completions, for instance with respect to $L^1(L)$ or the full $\Cs$-algebra $\Cs(L)$. However, the estimate \eqref{estimscal} on which the proof of Proposition \ref{charactE(G/N)} relies does not allow to exhibit $\Stz(X)$ as a submodule for these norms.
\end{remark}

The universal principal series $\E(G/\Nbar)$ associated to the pair $(G,\bar{P})$ can be described in the same way. More precisely, all the arguments of the above discussion hold, working on $\Xbar$ and using the following formulas for the actions and the inner product:

\begin{equation}\label{actionGbar}\tag{$\overline{\mbox{\ref{actionG}}}$}g.f(y)=f(yg),\end{equation}

\begin{equation}\label{actionLbar}\tag{$\overline{\mbox{\ref{actionL}}}$}f.a(y) = \left|\det(a)\right|^{\frac{n+1}{2}}f(ay) = \left|\det(a)\right|^{\frac{n+1}{2}}\;{^a}f(y),\end{equation}

\begin{equation}\label{actionLbarint}\tag{\ref{actionLbar}'}f.\varphi(y) = \int_{\mathrm{GL}(n,F)}f(ay)\left|\det(a)\right|^{\frac{n+1}{2}}d^\times a,\end{equation}

\begin{equation}\label{scalbar}\tag{$\overline{\mbox{\ref{scal}}}$}\scal{f}{h}_{\Xbar}(a) = \left|\det(a)\right|^{-\frac{n+1}{2}}\int_{\Xbar}\overline{f(y)}h(a^{-1}y)\,dy = \left|\det(a)\right|^{-\frac{n+1}{2}}\scal{f}{{^{a^{-1}}}h}_{L^2},\end{equation}
for $f$ and $h$ functions on $\Xbar$, $g$ in $G$, $a$ in $L$ and $\varphi$ in $C_c(L)$. As in Proposition \ref{charactE(G/N)}, the inclusion map $\bar{\iota}:C_c^\infty(\mathrm{M}^\times_{n,n+1}(F))\longrightarrow\E(G/\Nbar)$ extends to $\Stz\left(\Xbar\right)$.

\section{Intertwining operators}\label{secintertw}

The purpose of this section is to study and normalize standard intertwining integrals at the level of the Hilbert modules introduced above.

\subsection{Standard intertwining integrals}

The study of intertwiners between principal series induced from $P$ and $\bar{P}$ respectively, relies on the construction of operators transforming $N$-invariant functions into $\Nbar$-invariant ones (see \cite{KS1,KS2,Knapp1}). Working formally, that is outside of convergence considerations, one is naturally led to consider integrals of the form \begin{equation}\label{I}\I(f)\left(g\Nbar\right)=\int_{\Nbar}f(g\bn)\,d\bn.\end{equation}
It was observed (see \textit{e.g.} \cite{KobayashiMano_Schrodinger}) that such operators can be interpreted as Radon transforms and understood in the context of double fibrations as in \cite{Helgason_Radon}. The present situation is described by the diagram
\begin{equation}\label{doublefib}\xymatrix{&G\ar[dr]^{p_{\Nbar}}\ar[dl]_{p_N}&\\G/N&&G/\Nbar}\end{equation}
where $p_N$ and $p_{\Nbar}$ denote the natural projections. Lemma 6.2 in \cite{PArtmodules} implies that \eqref{I} defines a map $\I:C_c(G/N)\longrightarrow C(G/\Nbar)$ which does not extend to an operator between Hilbert modules, the latter fact being directly related to reducibility phenomena occurring in the principal series \cite[Appendix]{PC*entrelacSL2}.

In the notations of Diagram \eqref{doublefib}, the standard intertwining integral can be written as \[\I(f)(y)=\int_{p_N\left(p_{\Nbar}^{-1}(y)\right)}f(x)\,dx.\] In order to study this integral at the level of Hilbert modules, it will be convenient to work in $X$ and $\Xbar$ rather than $G/N$ and $G/\Nbar$, so \eqref{doublefib} will be replaced by
\begin{equation}\tag{\ref{doublefib}'}\xymatrix{&G\ar[dr]\ar[dl]&\\X&&\Xbar\;.}\end{equation}
First, we observe that two elements $x$ and $y$ in $X$ and $\Xbar$ come from the same $g$ in $G$ if and only if $x=gx_0$ and $y=y_0g^{-1}$, with the notations used in the proof of Proposition \ref{G/N}. Denoting by $\delta$ the Dirac distribution on $\mathrm{M}_n(F)$, it follows that the integral operator $\I$ is given by the distributional kernel $k_{\I}$ defined on $X\times\Xbar$ by \[k_{\I}(x,y)=\delta(yx - I_n),\] in the sense that \[\I(f)(y)=\int_Xk_{\I}(x,y)f(x)\,dx\] for any suitable function $f$, that is \begin{equation}\label{defIker}\I(f)(y)=\int_X\delta(yx - I_n)f(x)\,dx.\end{equation}

The following properties of $k_{\I}$ will be useful in the normalization process studied in the next section. Formulas \eqref{actionL} and \eqref{actionLbar} imply that $k_{\I}$ is invariant under the diagonal $L$-action: \begin{equation}k_{\I}(x.a,y.a)=k_{\I}(x,y)\end{equation} for $a$ in $\mathrm{GL}_n(F)$. Moreover, the homogeneity property of the Dirac distribution implies that \begin{equation}\label{kIhomog}k_{\I}(x.a,y)=\left|\det(a)\right|^{-(n+1)}\delta(yx-a).\end{equation}

Our main purpose will be to construct a $G$-equivariant unitary operator between Hilbert modules that differs from $\I$ only by a convolution over $L$. In order to state a more precise definition, let us introduce one more notation.

\begin{definition}
If $T$ is a function on $L$, the operator $C_T$ of convolution by $T$ acts on a function $f$ on $G/N$ by \[C_T(f)(x)=f*_L T(x)=\int_L f.a(x)T(a)\,da.\]
\end{definition}

In the present situation, \eqref{actionLint} implies that \begin{equation}\label{convolT}C_T(f)(x)=\int_{\mathrm{GL}(n,F)}f\left(xa^{-1}\right)\left|\det(a)\right|^{-\frac{n+1}{2}}T(a)\,d^\times a.\end{equation}

\begin{definition}\label{normaldef}A bounded operator $\U:\E(G/N)\longrightarrow\E(G/\Nbar)$ is said to \emph{normalize} the standard intertwining integral $\I$ defined by \eqref{defIker} if \begin{enumerate}[($i$)]
\item it is unitary and $G$-equivariant;
\item there exists a function $\gamma$ on $L$ such that $\left(\I\circ C_\gamma\right)(f)=\U(f)$ for $f$ in a dense subspace of $\E(G/N)$.
\end{enumerate}
\end{definition}

\begin{remark}\label{remSL2}Definition \ref{normaldef} above shoud be compared with Definition 3.1 in \cite{PC*entrelacSL2} where normalization is defined for standard intertwining integrals associated to Weyl elements.\end{remark}

\subsection{\texorpdfstring{The Fourier transform on $X$}{The Fourier transform on X}}

The use of the Fourier transform on matrix spaces to study principal series representations goes back to the work of E. M. Stein \cite{SteinSLn}. Let us briefly collect here the main properties of the integral operator that will give rise to the $\Cs$-algebraic intertwiner studied in the next section.

The transform defined by 

\begin{equation}\label{F}\F f(y)=\int_Xf(x)e^{-2i\pi\re \left(\mathrm{Tr}(yx)\right)}dx\end{equation}

maps $\Stz(X)$ to $\Stz(\Xbar)$. It satisfies the relation

\begin{equation}\label{equivarF}\F\left(f^{a^{-1}}\right)=\left|\det(a)\right|^{-(n+1)}\: ^{a}\left(\F(f)\right).\end{equation}

\subsection{Normalization}\label{normalsubsec}

We can now establish the main result of the article, relating the Fourier transform of the previous paragraph to the standard intertwining operator at the level of the $\Cs$-algebraic universal principal series.

\begin{theorem}\label{thmnormal}
The transform $\F$ between $\Stz(X)$ and $\Stz\left(\Xbar\right)$ extends to a unitary operator of Hilbert modules \[\U:\E(G/N)\longrightarrow\E(G/\Nbar)\] that normalizes $\I$ in the sense of Definition \ref{normaldef}. The corresponding normalizing distribution is defined by \[\gamma_n(a)=\left|\det(a)\right|^{\frac{1-n}{2}}e^{-2i\pi\re \left(\mathrm{Tr}(a^{-1})\right)}\] for $a$ in $L$.
\end{theorem}

\begin{proof}
The $G$-equivariance is a direct consequence the $G$-invariance of the measure. To establish that $\F$ is $C_c(L)$-linear, it is enough to show that it is $L$-equivariant.

In view of Proposition \ref{charactE(G/N)} and Remark \ref{remE(G/N)}, we work at the level of Schwartz functions on $X$. For $f$ in $\Stz(X)$ one has, according to \eqref{actionL},
\[\begin{array}{cclr}\F(f.a) & = & \F\left(\left|\det(a)\right|^{-\frac{n+1}{2}}f^{a^{-1}}\right)&\\
& = & \left|\det(a)\right|^{\frac{n+1}{2}}.\;{^a}\F(f) & \text{by \eqref{equivarF}}\phantom{.}\\
& = & \F(f).a & \text{by \eqref{actionLbar}}.
\end{array}\]
The fact that $\F$ preserves the $\Csr(L)$-valued inner products also follows from a straightforward computation relying on its equivariance properties, namely \begin{eqnarray*}\scal{\F (f)}{\F (f)}_{\Xbar}(a) & = & \left|\det(a)\right|^{-\frac{n+1}{2}}\scal{\F (f)}{{^{a^{-1}}}\F (f)}_{L^2}\\
& = & \left|\det(a)\right|^{-\frac{n+1}{2}}\scal{\F (f)}{\left|\det(a)\right|^{(n+1)}\F \left(f^a\right)}_{L^2}\\
& = & \left|\det(a)\right|^{\frac{n+1}{2}}\scal{\F (f)}{\F \left(f^a\right)}_{L^2}\\
& = & \left|\det(a)\right|^{\frac{n+1}{2}}\scal{f}{f^a}_{L^2} = \scal{f}{f}_X(a),\end{eqnarray*}
the last line relying on the Plancherel equality for square integrable functions.

Let now $\gamma$ be a function on $L$. According to \eqref{convolT}, the composition $\I\circ C_\gamma$ acts on a function $f$ by \begin{eqnarray*}\left(\I\circ C_\gamma\right)(f)(y)&=&\int_X \int_{\mathrm{GL}(n,F)}k_{\I}(x,y)f\left(xa^{-1}\right)\left|\det(a)\right|^{-\frac{n+1}{2}}\gamma(a)\,d^\times a\,dx\\
&\overset{x\leftrightarrow xa}{=}&\int_X \int_{\mathrm{GL}(n,F)}k_{\I}(xa,y)f(x)\left|\det(a)\right|^{\frac{n+1}{2}}\gamma(a)\,d^\times a\,dx
\end{eqnarray*}

Therefore, the kernel $k_\circ$ defining the composition can be written as

\begin{eqnarray*}
k_\circ(x,y)&=&\int_{\mathrm{GL}(n,F)}k_{\I}(xa,y)\left|\det(a)\right|^{\frac{n+1}{2}}\gamma(a)\,d^\times a\\
&\overset{\quad a\leftrightarrow a^{-1}}{=}&\int_{\mathrm{GL}(n,F)}k_{\I}(xa^{-1},y)\left|\det(a)\right|^{-\frac{n+1}{2}}\gamma(a^{-1})\,d^\times a\\
&=&\int_{\mathrm{GL}(n,F)}\delta(yx-a)\left|\det(a)\right|^{\frac{n+1}{2}}\gamma(a^{-1})\,d^\times a\qquad\text{by \eqref{kIhomog}}\\
&=&\int_{\mathrm{GL}(n,F)}\delta(yx-a)\left|\det(a)\right|^{\frac{1-n}{2}}\gamma(a^{-1})\,da
\end{eqnarray*}

Comparing this last expression with the kernel of the Fourier transform of the previous paragraph \[k_{\F}(x,y)=\int_{\mathrm{M}_n(F)}\delta(yx-a)e^{-2i\pi\re \left(\mathrm{Tr}(a)\right)}da,\]
it is easily seen that the result follows from choosing \[\gamma_n(a)=\left|\det(a)\right|^{\frac{1-n}{2}}e^{-2i\pi\re \left(\mathrm{Tr}(a^{-1})\right)}.\]
\end{proof}

Pursuing the comparison with the case of $\mathrm{SL}(2)$ suggested in Remark \ref{remSL2}, we observe that this result is consistent with the one obtained in the case $n=1$ in \cite[Theorem 3.2]{PC*entrelacSL2}, where the composition $\I_w\circ\F_w$ was proved to be equal to a convolution operator by a distribution analogous to $\gamma_1$.

\section{The non-archimedean case}\label{secnonarchi}

The purpose of this last paragraph is to describe how the construction and the main result presented above extend to the non-archimedean case.

Let $F$ be a non-archimedean local field, that is a finite extension of $\Q_p$ or $\Fp((t))$. We fix a non-trivial continuous character $\chi:F\longrightarrow\mathrm{U}(1)$ and a Haar measure $dx$ on $F$ so that the Fourier transform $\F_\chi$ defined by \[\F_\chi(f)(y)=\int_F f(x)\chi(xy)\,dx\] is an isometry of $L^2(F)$.

A norm $\left|\cdot\right|_F$ is fixed on $F$ that coincides with the modular function related to the Haar measure (see \cite{WeilBNT}). It is specified by the relation $\left|\pi\right|_F=q^{-1}$ where $\pi$ is a uniformizer and $q$ is the cardinality of the residual field.

The structure theory of the group $G_F$ remains the same in this context. The definitions of functional spaces in the case of a totally discontinuous space $T$ need to be adapted as follows: $C_c(T)$ still denotes continuous functions with compact support on $T$, while $C^\infty_c(T)$ (and $\Stz(T)$) denote locally constant compactly supported functions on $T$. Up to these modifications and replacing absolute values by $\left|\cdot\right|_F$ everywhere, the Hilbert module construction of Section \ref{E(G/N)sec} providing the reduced $\Cs$-algebraic universal principal series can be carried out in the same way.

Finally, the proof of the analogue of Proposition \ref{charactE(G/N)} is straightforward and the normalisation process can be achieved by defining the Fourier transform $\F$ between $\Stz(X)$ and $\Stz(\Xbar)$ by
\begin{equation}\label{U_n_chi}\F f(y)=\int_Xf(x)\chi\left(\mathrm{Tr}(yx)\right)\,dx\end{equation} then extending it as an operator of Hilbert modules $\E(G/N)\longrightarrow\E(G/\Nbar)$ as in Theorem \ref{thmnormal}. The normalising function $\gamma_n$ is then given by \begin{equation}\label{gamma_n_chi}\gamma_n(a)=\left|\det(a)\right|_F^{\frac{1-n}{2}}\chi \left(\mathrm{Tr}(a^{-1})\right)\end{equation} for any $a$ in $L$.

\subsection*{Concluding remarks}

The point of view on principal series and intertwining operators presented here relates to other work on the subject. More precisely, the distribution $\gamma_n$ in \eqref{gamma_n_chi} allows to recover local $\gamma$ factors introduced by R. Godement and H. Jacquet in \cite{GodJac} and the Hilbert module operator obtained from \eqref{U_n_chi} by Theorem \ref{thmnormal} is a way of considering simultaneously a whole family of the normalised intertwiners studied by F. Shahidi in \cite{Shahidi84}. The connection with these results will be studied in detail in future work.

\newpage

\bibliographystyle{amsalpha}
\bibliography{biblio}

\end{document}